\documentclass{article}
\usepackage{college-math-j}
\usepackage{hyperref}

\theoremstyle{theorem}
\newtheorem{theorem}{Theorem}

\theoremstyle{definition}

\newtheorem*{remark}{Remark}

\begin{document}

\section{A Short and Elementary Proof of the Two-sidedness of the Matrix-Inverse}
Pietro Paparella (pietrop@uw.edu) University of Washington Bothell  

\bigskip
\noindent 

Beauregard \cite{b2007} and Strang \cite[pp.~88-89]{s2016} give an elementary proof of the two-sidedness of the matrix-inverse using elementary matrices. Sandomierski \cite{s2012} gives another elementary proof using an under-determined homogeneous linear system. Hill \cite{h1967} provides a proof using basic results of rings and ideals, and Baksalary \& Trenkler \cite{bt2013} provide a proof using the Moore-Penrose inverse.  

However, as we shall see, the two-sidedness of the matrix inverse can be shown using only the notions of \emph{linear independence} and the \emph{reduced row-echelon form} of a matrix, concepts that feature prominently in an elementary linear algebra course (it is worth noting that the proof of \hyperref[twoside]{Theorem \ref*{twoside}} is in the same vein as that found in the advanced textbook by Meyer \cite[Example 3.7.2]{m2000}).

In addition, and as an application of this result, we prove that a matrix is invertible if and only if it is row-equivalent to the identity matrix without appealing to elementary matrices, which is common in textbooks (see, e.g., \cite{llm2016,s2016}). This result is shown by utilizing the feasibility of certain linear systems and appealing to a well-known theorem. 

Such demonstrations are valuable pedagogically, given that the linear system is the central object of study in a typical elementary linear algebra course. 

\section{The Proofs}
In what follows all matrices are assumed to be $n$-by-$n$ and all vectors are $n$-by-$1$. As is standard, the identity matrix is denoted by $I$. %Our demonstration also underscores the importance of a \textit{basis} and provides a proof of the \textit{Invertible Matrix Theorem}.

%\begin{definition}
%An $n$-by-$n$ matrix $A$ is called \emph{invertible} if there is an $n$-by-$n$ matrix $B$, called the \emph{inverse} of $A$, such that $AB = I$. 
%\end{definition}

\begin{theorem}
\label{twoside}
If $A$ and $B$ are matrices such that $AB= I$, then $BA = I$.
\end{theorem}

\begin{proof}
First, we show that the columns of $B$ form a linearly independent set; to the contrary, if $Bx = 0$, where $x \neq 0$, then $x = Ix = ABx= A(Bx) = A(0) = 0$,  
a contradiction.

Because the columns of $B$ form a linearly independent set, it follows that the reduced row-echelon form of $B$ contains a pivot in every column and hence every row; following a well-known theorem \cite[Theorem 4, \S 1.4]{llm2016}, the equation $Bx = b$ has a solution for every vector $b$. 

Next, we argue that the columns of $A$ must form a linearly independent set; to that end, suppose $Ay = 0$, where $y \neq 0$. From previous, there is a necessarily nonzero vector $x$ such that $y = Bx$. Thus, $0 = Ay = A(Bx) = (AB)x = Ix = x$, a contradiction. 

Multiplying both sides of the equation $AB=I$ on the right by $A$ yields $ABA = A$. Subtracting $A$ from both sides and applying the distributive property yields $A(BA- I) = 0$. 

We claim that the matrix $Z:=BA - I = 0$; if not, then $Z$ contains a nonzero entry and hence a nonzero column -- say the $j\textsuperscript{th}$-column, which we denote by $z_j$. Thus, $Az_j = 0$, which is a contradiction unless $Z = BA - I = 0$, i.e., $BA = I$.
\end{proof}

\begin{theorem}
A matrix $A$ is invertible if and only if $A$ is row-equivalent to $I$.
\end{theorem}

\begin{proof}
Following \hyperref[twoside]{Theorem \ref*{twoside}}, $A$ is invertible if and only if there is a  matrix $X$ such that $AX = I$. Thus, $A$ is invertible if and only if the linear system $Ax = e_i$ is feasible for all $i \in \{1, \dots, n\}$, where $e_i$ denotes the $i\textsuperscript{th}$ column of $I$. 

We claim that if $Ax = e_i$ is feasible for all $i \in \{1,\dots,n\}$, then $Ax = b$ is feasible for every vector $b$ (note that the converse of this statement is obviously true). To that end, if  $s_i$ is any solution to $Ax = e_i$ and $b= [ b_1~\cdots~b_n ]^\top$, then, following the linearity of $A$, 
\begin{align*} 
b= I b := \sum_{i=1}^n b_i e_i = \sum_{i=1}^n b_i ( As_i) = \sum_{i=1}^n A (b_i s_i) = A \left( \sum_{i=1}^n b_i s_i \right),
\end{align*}
and the claim is established.

Since $Ax = b$ for every vector $b$, another application of \cite[Theorem 4, \S 1.4]{llm2016} reveals that $A$ is invertible if and only if the reduced row-echelon form of $A$ possesses a pivot in every row, i.e., $A$ is invertible if and only if it is row-equivalent to $I$. 
\end{proof}

\begin{remark}
Our demonstration proves the equivalence of the following statements of the \textit{invertible matrix theorem}: 
\begin{itemize}
\item $A$ is invertible.
\item $A$ is row-equivalent to $I$.
\item $Ax = b$ is feasible for every vector $b$.
\item the reduced row-echelon form of $A$ has a pivot in every row.
\item the columns of $A$ are linearly independent.
\item the equation $Ax = 0$ has only the trivial solution.
\item the columns of $A$ span $\mathbb{R}^n$.
\item the columns of $A$ form a basis for $\mathbb{R}^n$.
\end{itemize}
\end{remark}

%\begin{acknowledgment}
%The authors wish to thank the Greek polymath Anonymous, whose prolific works are an endless source of inspiration.
%\end{acknowledgment}

\begin{abstract}
An elementary proof of the two-sidedness of the matrix-inverse is given using only linear independence and the reduced row-echelon form of a matrix. In addition, it is shown that a matrix is invertible if and only if it is row-equivalent to the identity matrix without appealing to elementary matrices. This proof underscores the importance of a \textit{basis} and provides a proof of the \textit{invertible matrix theorem}.
\end{abstract}

\vfill\eject


\begin{thebibliography}{100}
\bibitem{bt2013} O.~M.~Baksalary and G.~Trenkler, Another Proof of the Two-sidedness of Matrix Inverses, IMAGE, \textbf{50} (2013) 2.

\bibitem{b2007} R.~A.~Beauregard, A short proof of the two-sidedness of matrix inverses, \emph{Math. Mag.} \textbf{80} (2) (2007) 135--136.   

\bibitem{h1967} P.~Hill, On the matrix equation $AB= I$, \emph{American Math. Montly} \textbf{74} (7) (1967) 848--849. DOI: 10.2307/2315819

\bibitem{llm2016} D.~C.~Lay, S.~R.~Lay, and J.~J.~McDonald, \emph{Linear Algebra and Its Applications}, 5th ed., Pearson Education, 2016.

\bibitem{m2000} C.~D.~Meyer, \emph{Matrix Analysis and Applied Linear Algebra}, SIAM, Philadelphia, PA, 2000.

\bibitem{s2012} F.~Sandomierski, An elementary proof of the two-sidedness of matrix inverses, \emph{Math. Mag.} \textbf{85} (4) (2012) 289. DOI: 10.4169/math.mag.85.4.289

\bibitem{s2016} W.~G.~Strang, \emph{Introduction to Linear Algebra}, 5th ed., Wellesley-Cambridge, Wellesley, MA, 2016.
\end{thebibliography}
\end{document}